\documentclass{amsart}

\usepackage{amsmath, amssymb, amsthm} 
\usepackage{braket} 
\usepackage{enumerate}
\usepackage[dvipdfmx]{graphicx}

\newtheorem{thm}{Theorem}[section]

\newtheorem{prop}[thm]{Proposition}

\theoremstyle{definition}
\newtheorem{dfn}{Definition}[section]
\newtheorem{rmk}[dfn]{Remark}
\newtheorem{exmp}[dfn]{Example}

\newcommand{\conj}[1]{|#1|} 
\numberwithin{equation}{section}
{\begin{center}\begin{minipage}[t]{.8\hsize} \hrule \ovalbox{TODO} #1}%
{\smallskip\hrule \end{minipage}\end{center}\smallskip}%

\title[Partitions of cyclic words]{%
Partitions of cyclic words and Goldman-Turaev Lie bialgebra}
\author{Ryosuke YAMAMOTO}
\address{Faculty of Education, Gunma University, 4-2 Aramaki-machi, Maebashi, Gunma, 371-8510 Japan}
\email{yamryo0202@gunma-u.ac.jp}

\begin{document}

\begin{abstract}
 The free $\mathbb{Z}$-module generated from the set of non-trivial homotopy classes of
 closed curves on an oriented surface has the structure of
 Lie bialgebra by two operations, Goldman bracket and Turaev cobracket.
 M.\ Chas gave a combinatorial redefinition of these operations
 through the identification of the homotopy classes of closed curves
 on the surface
 with the cyclic words associated with the word representation of elements
 of the fundamental group of the surface.
 We present a new approach to give a combinatorial definition of the bracket
 and cobracket, focusing on the information given by the partitions of cyclic words.
\end{abstract}

\maketitle

\section{Introduction}\label{sec:intro}
 Let $\Sigma=\Sigma_{g,1}$ be a compact oriented surface of genus $g \geq 1$
 with a connected boundary,
 and $\hat{\pi}$ the set of (free) homotopy classes of oriented closed curves on $\Sigma$.
 The free $\mathbb{Z}$-module generated from $\hat{\pi}$
 has a structure of Lie algebra
 by an operation called the Goldman bracket
 \cite{goldman1986invariant}.
 Let $1 \in \hat{\pi}$ be the homotopy class of the trivial loop.
 The quotient Lie algebra $\mathbb{Z}{\hat{\pi}}/\mathbb{Z}1$
 has an operation called the Turaev cobracket
 which gives a structure of Lie bialgebra on
 $\mathbb{Z}{\hat{\pi}}/\mathbb{Z}1$ \cite{turaev1991skein}.

 These two operations are deeply related to
 the geometric intersection number of loops on a surface,
 since they both are defined from the intersections of representatives
 of homotopy classes of loops and
 are independent of the choice of the representatives.
 In fact, these operations determine the geometric intersection number of
 certain types of loops, as shown in \cite{chas2004combinatorial}, \cite{chas2015algebraic}.
 In these works
 a ``combinatorial'' definitions of the bracket and cobracket
 introduced by M.\ Chas are used.
 ``Combinatorial'' means as follows:
 The elements of $\hat{\pi}$ can be represented as reduced cyclic words
 (we will recall in the next section).
 Chas redefined the bracket and cobracket on the reduced cyclic words
 using information of the letters of the words.

In this paper, we will discuss a new approach of the combinatorial definition
of the Goldman bracket and the Turaev cobracket.
We will focus on the partitions of cyclic words and
introduce the notion of linking number for a pair of partitions of words.
We then define an operation on (not necessary reduced) cyclic words
which yields a sum of elements in $\mathbb{Z}\hat{\pi}\otimes\mathbb{Z}\hat{\pi}$
such that each element is given from a pair of partitions of a word
with the linking number of the partitions as coefficient.
We will see that this operation gives an new definition of the Turaev cobracket.
The Goldman bracket is redefined in a similar manner.

 This paper is organised as follows:
 In \S \ref{sec:prelim} we will recall the original definitions of
 the Goldman bracket and the Turaev cobracket,
 and prepare some notions about cyclic words
 representing the homotopy classes of free loops on $\Sigma$.
 In \S \ref{sec:linking}
 we organise the connection between the letters and the partitions
 of a cyclic word in the form of a chain complex,
 and introduce the notion of linking number of a pair of partitions of a cyclic word
 (or cyclic words)
 which gives information of the self-intersections
 of a loop on $\Sigma$ corresponding to the word
 (or the intersections of loops on $\Sigma$ corresponding to the words).
 Finally in \S \ref{sec:CGT} we give our combinatorial definitions of
 the Turaev cobracket and the Goldman bracket.
 In addition, we will see that one can confirm the well-definedness of
 the operations as a map on $\mathbb{Z}\hat{\pi}/\mathbb{Z}1$ directly from our definitions.

\section{Preliminaries}\label{sec:prelim}
In this section we first recall the original definitions of two operations of
Goldman-Turaev Lie bialgebra,
and then review the notion of the cyclic words corresponding
to the free homotopy classes of loops on $\Sigma$.

\subsection{Goldman bracket}
 For two homotopy classes of oriented loops $\alpha$ and $\beta$ in $\hat\pi$,
 take representatives $\tilde{\alpha}$ and $\tilde{\beta}$ respectively
 so that $\tilde{\alpha}$ and $\tilde{\beta}$ are in general position,
 i.e., their intersection points are all transversal double points.
 Let $\Gamma(\tilde{\alpha}, \tilde{\beta})$ be the set of all intersection points
 of $\tilde{\alpha}$ with $\tilde{\beta}$,
 and for each $p \in \Gamma(\tilde{\alpha}, \tilde{\beta})$,
 let $\tilde{\alpha}_{p}$ ($\tilde{\beta}_{p}$)
 the loop $\tilde{\alpha}$ ($\tilde{\beta}$ respectively)
 viewed as an oriented loop based at $p$.
 Then the bracket of $\alpha$ and $\beta$ is defined as follows;
 \begin{equation}
  [\alpha, \beta] := \sum_{p \in \Gamma(\tilde{\alpha}, \tilde{\beta})}
   \operatorname{sign}(\tilde{\alpha}_{p}, \tilde{\beta}_{p})
   |\tilde{\alpha}_{p}\tilde{\beta}_{p}|,
 \end{equation}
 where
 $\operatorname{sign}(\tilde{\alpha}_{p}, \tilde{\beta}_{p})$ is $+1$
 if the pair of the tangent vector of $\tilde{\alpha}$ and of $\tilde{\beta}$ at $p$
 form the positive basis of the tangent plane of $\Sigma$ at $p$ and is $-1$ otherwise,
 and $|\tilde{\alpha}_{p}\tilde{\beta}_{p}|$ denotes the (free) homotopy class of
 the based loop $\tilde{\alpha}_{p}\tilde{\beta}_{p}$.
 %
 In \cite{goldman1986invariant}, Goldman proved well-definedness of this operation.
 Linearly expanding this bracket onto $\mathbb{Z}{\hat{\pi}}$,
 we obtain the Goldman bracket
 $[,]: \mathbb{Z}{\hat{\pi}}\otimes \mathbb{Z}{\hat{\pi}} \rightarrow \mathbb{Z}{\hat{\pi}}$.
 \medskip

\subsection{Turaev cobracket}
 For a homotopy classes of oriented loops $\alpha$ in $\hat\pi$,
 take a representative $\tilde{\alpha}$ in general position.
 Let $\Gamma(\tilde{\alpha})$ be the set of all self-intersection points
 of $\tilde{\alpha}$,
 and for each $p \in \Gamma(\tilde{\alpha})$,
 let $\tilde{\alpha}_{p}^{1}$ and $\tilde{\alpha}_{p}^{2}$
 the two oriented loops based at $p$ obtained by dividing $\tilde{\alpha}$ at $p$.
 (Either one can be $\tilde{\alpha}_{p}^{1}$.)
 Then we may define the following operation;
 \begin{equation}
  \Delta(\alpha) := \sum_{p \in \Gamma(\tilde{\alpha})}
   \operatorname{sign}(\tilde{\alpha}_{p}^{1}, \tilde{\alpha}_{p}^{2})
   \left(
    |\tilde{\alpha}_{p}^{1}|\otimes|\tilde{\alpha}_{p}^{2}|
    -|\tilde{\alpha}_{p}^{2}|\otimes|\tilde{\alpha}_{p}^{1}|
   \right),
 \end{equation}
 where
 $\operatorname{sign}(\tilde{\alpha}_{p}^{1}, \tilde{\alpha}_{p}^{2})$
 is the same as above,
 and $|\tilde{\alpha}_{p}^{1}|$ ($|\tilde{\alpha}_{p}^{2}|$)
 also denotes the (free) homotopy class of
 the based loop $\tilde{\alpha}_{p}^{1}$ ($\tilde{\alpha}_{p}^{2}$ respectively).
 %
 %
 Since the value $\Delta(\alpha) \in \mathbb{Z}{\hat{\pi}}$
 has an ambiguity by $1 \in \mathbb{Z}{\hat{\pi}}$,
 we need to take the quotient $\mathbb{Z}{\hat{\pi}}/\mathbb{Z}1$
 to have well-definedness of the operation.
 Linearly expanding this operation onto $\mathbb{Z}{\hat{\pi}}/\mathbb{Z}1$,
 we obtain the Turaev cobracket
 $[,]: \mathbb{Z}{\hat{\pi}}/\mathbb{Z}1 \rightarrow
 \mathbb{Z}{\hat{\pi}}/\mathbb{Z}1\otimes \mathbb{Z}{\hat{\pi}}/\mathbb{Z}1$
  \cite{turaev1991skein}.

\subsection{Words and cyclic words}
Let $\{a_{k}, b_{k}\}_{k=1,2,\dots, g}$ be
a set of $2g$ based loops on $\Sigma$
such that
they are mutually disjoint except for the base-point,
mutually non-parallel
and each pair $(a_{k}, b_{k})$ ($1 \geq k \geq g$)
gives a one-handle part of a handle decomposition of $\Sigma$.
We will also denote by $a_{k}$ (resp.\ $b_{k}$) the based homotopy class of
the based loop $a_{k}$ (resp.\ $b_{k}$).
The fundamental group of the surface $\Sigma$, we will denote it by $\pi$,
is identified with the free group of rank $2g$
with respect to the generators $\{a_{k}, b_{k}\}_{k=1,\dots, g}$.

On the free group $\pi$, the alphabet is the set
$\mathcal{A}=\{a_{k}, a_{k}^{-1}, b_{k}, b_{k}^{-1}\}_{k=1,\dots,g}$,
and the letters are the elements of $\mathcal{A}$.
We denote by $W(\mathcal{A})$ be the set of words on $\mathcal{A}$.
The contraction (resp.\ the cyclic contraction) on a word in $W(\mathcal{A})$ is
an operation that removing successive two letters (resp.\ the first and the last letter)
if they are inverse each other.
The cyclic words on $\mathcal{A}$ are the equivalent classes of $W(\mathcal{A})$
with respect to the equivalence relation
``becomes the same word by contractions and cyclic contractions''.
Throughout this paper we will regard that the words in $W(\mathcal{A})$
as representatives of cyclic words
by considering that the last letter of a word is followed by the first letter.

\section{Partitions of words and their linking numbers}\label{sec:linking}
In this section
we discuss a correspondence between sub-arcs of a free loop and
partitions of the word associated with the loop.
We then introduce the notion of linking numbers of a pair of petitions of words.
%
In the rest of this paper we take the base-point on $\partial\Sigma$ and
generator loops $\{a_{1}, b_{1}, \dots, a_{g}, b_{g}\}$ of $\pi=\pi_{1}(\Sigma)$
as shown in Figure \ref{161552_15Apr21}.
\begin{figure}[h]
 {\unitlength=1mm
 \begin{picture}(0,0)(0,0)
  \put(5,10){$a_{1}$}\put(26,3){$b_{1}$}
  \put(68,10){$a_{g}$}\put(90,3){$b_{g}$}
 \end{picture}}
 \includegraphics[width=.8\hsize, keepaspectratio]{%
 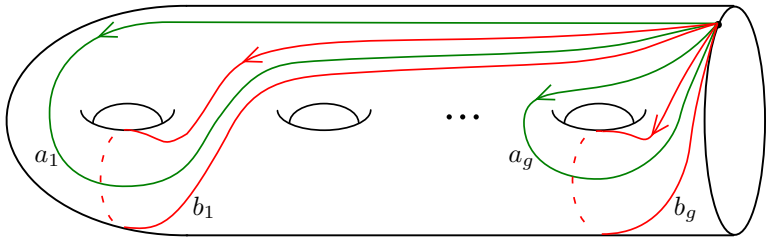} 
 \caption{generators of $\pi$}
 \label{161552_15Apr21}
\end{figure}
Recall that the based loops on $\Sigma$ are represented by the words in $W(\mathcal{A})$,
which consist of the letters of the alphabet
$\mathcal{A}=\{a_{k},a_{k}^{-1},b_{k},b_{k}^{-1}\}_{1\leq k \leq g}$.

\subsection{Partitions of a word}
 We consider the partitions of a word each of which splits the word in two.
 For a word $w=x_{1}x_{2}\dots x_{p}$ in $W(\mathcal{A})$,
 we call the partition between $x_{i}$ and $x_{i+1}$
 the $i$-th partition of $w$, and denote it by $\phi_{i}w$.
 We also consider the $p$-th partition $\phi_{p}w$
 between $x_{p}$ and $x_{1}$ so that
 we may regard the word as representative of a cyclic word.

\subsection{Arc diagram of a word}
We will introduce ``arc diagram'' of a word in $W(\mathcal{A})$
to associate the partitions of the word
with the self-intersections of the corresponding loop on $\Sigma$.

%
Let $D$ be a small disk neighbourhood of the base-point on $\Sigma$,
and $\partial_{0} D$ the (open) subarc of $\partial D$
which is the intersection of $\partial D$ with the interior of $\Sigma$.
We assume that any of the generator loops $a_{1},b_{1},\dots ,a_{g},b_{g}$
intersects just twice with $\partial_{0} D$,
We call the first intersection point at which a generator loop crosses $\partial_{0} D$
while following the loop from the base-point along its direction,
 the gate of the loop,
 and consider that the another intersection point of the generator loop with $\partial_{0} D$
 is the gate of the inverse of the generator loop.
We give the orientation on $\partial_{0} D$
that matches with the direction from the gate of $a_{1}$ to the gate of $b_{1}$.
For two distinct gates $x$ and $y$ on $\partial_{0} D$,
we express as $x<y$ if the direction from $x$ to $y$ along $\partial_{0} D$
coincides the orientation of $\partial_{0} D$,
and as $x>y$ otherwise.
%
Relating to the order of the gates of the generator loops
and their inverses,
we give the following order on the alphabet
$\mathcal{A}$;
\begin{equation*}
 a_{1} < b_{1} < a_{1}^{-1} < b_{1}^{-1} <
 a_{2} < b_{2} < a_{2}^{-1} < b_{2}^{-1} < \dots <
 a_{g} < b_{g} < a_{g}^{-1} < b_{g}^{-1}.
\end{equation*}
%

For each word $w=x_{1}x_{2}\dots x_{p} \in W(\mathcal{A})$,
we will fix a way of choosing a based loop on $\Sigma$ represented by $w$
as follows, and denote the chosen loop by $\ell_{b}(w)$:
\begin{enumerate}[i)]
 \item The based loop $\ell_{b}(w)$ is the union of based loops mutually disjoint
       except for the base-point each of which
       corresponds to each of the letters of $w$.
 \item Let the letter $x_{i}$ denote the gate of the sub-loop of $\ell_{b}(w)$
       corresponding $x_{i}$,
       and $x_{i}^{-1}$ the gate of the inverse of the sub-loop.
       The gates $x_{1}, \dots, x_{p}$
       and the gates $x_{1}^{-1}, \dots, x_{p}^{-1}$
       are lined up on $\partial_{0}D$ by the following rule:
       Two gates $x_{i}^{\varepsilon}$ and $x_{j}^{\delta}$
       ($i < j$, $\varepsilon, \delta \in \{-1, 1\}$)
       are in the order of $\mathcal{A}$ if these letters are distinct.
       In the case where these letters are the same,
       they are in order of $x_{i}^{\varepsilon} < x_{j}^{\delta}$
       if the letter is $a_{k}$ or $b_{k}$,
       and $x_{i}^{\varepsilon} > x_{j}^{\delta}$ otherwise.
\end{enumerate}

By altering $\ell_{b}(w)$ within $D$ so that each proper arc connects gates directly,
without touching the base-point,
we now obtain a free loop on $\Sigma$,
which is (freely) homotopic to $\ell_{b}(w)$.
We denote it by $\ell(w)$.

We call the disk $D$ with proper arcs $\ell(w) \cap D$
the arc diagram of $w$ (See Figure \ref{135131_15Dec21}(left)).
 Because of obvious one-to-one correspondence
 between the partitions of $w$ and the proper arcs on the arc diagram of $w$,
 we also denote by $\phi_{i}w$ the simple oriented proper arc
 from the gate $x_{i}^{-1}$ to the gate $x_{i+1}$.
 ($\phi_{p}w$ is the arc connecting $x_{p}^{-1}$ and $x_{1}$.)

  We also define the arc diagram for the pair of words
   to use for redefinition of the Goldman bracket in \S \ref{sec:CGT}.
  The arc diagram of a pair of words
  $v=x_{1}x_{2}\dots x_{p}$ and $w=y_{1}y_{2}\dots y_{q}$
  is the arc diagram of $v$ overlapping with the arc diagram of $w$
  such that
  $\partial D$, $\partial_{0} D$ and the base-point
  of two diagrams are identified each,
  and
  the gates of $v$ and the gates of $w$ are lined up on $\partial_{0} D$
  in the following order:
  A gate $x_{i}^{\varepsilon}$ of $v$ and a gate $y_{j}^{\delta}$ of $w$
  ($1\leq i \leq p$, $1\leq j \leq q$, $\varepsilon, \delta \in \{-1,1\}$)
  are in the order of $\mathcal{A}$ if their letters are distinct.
  In the case where these letters are the same,
  they are in order of $x_{i}^{\varepsilon} < y_{j}^{\delta}$
  if the letter is $a_{k}$ or $b_{k}$,
  and $x_{i}^{\varepsilon} > y_{j}^{\delta}$ otherwise
  (See Figure \ref{135131_15Dec21}(right)).
 Note that the order between a gate of $v$ and a gate of $w$ does not depend on
 their positions on each word.

\begin{figure}[h]
 {\unitlength=1mm
 \begin{picture}(50,30)(5,0)
  \put(0,0){\includegraphics[height=25mm]{%
  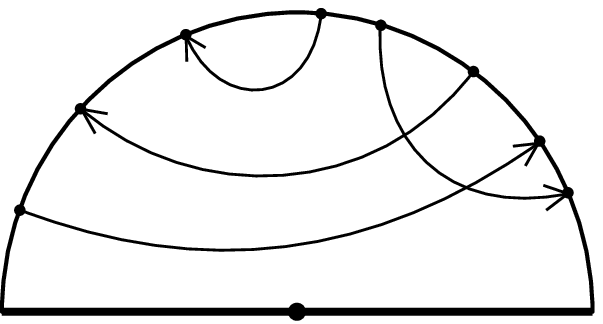}} 
  \put(47, 10){$a_{1}$}
  \put(45, 15){$a_{1}$}
  \put(38, 23){$b_{1}$}
  \put(30, 27){$a_{1}^{-1}$}
  \put(24, 28){$a_{1}^{-1}$}
  \put(12, 26){$b_{1}^{-1}$}
  \put(2, 19){$a_{2}$}
  \put(-4, 10){$a_{2}^{-1}$}
 \end{picture}
 \begin{picture}(50,30)(-10,0)
  \put(0,0){\includegraphics[height=25mm]{%
  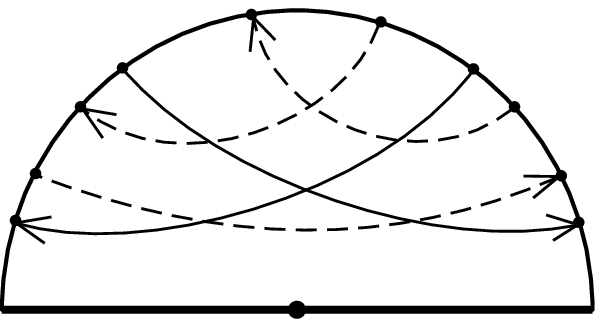}} 
  \put(48, 7){$a_{1}$}
  \put(47, 12){$a_{1}$}
  \put(43, 18){$a_{1}^{-1}$}
  \put(38, 23){$a_{1}^{-1}$}
  \put(30, 27){$b_{1}$}
  \put(18, 28){$b_{1}^{-1}$}
  \put(6, 23){$a_{2}$}
  \put(1, 18){$a_{2}$}
  \put(-4, 13){$a_{2}^{-1}$}
  \put(-6, 7){$a_{2}^{-1}$}
 \end{picture}
 }
 \caption{The arc diagram of a word $a_{1}b_{1}^{-1}a_{2}a_{1}$ (left)
 and the arc diagram of a pair of words $a_{1}a_{2}^{-1}$ and $a_{1}b_{1}^{-1}a_{2}$ (right).}
 \label{135131_15Dec21}
\end{figure}

\subsection{Linking number of partitions of words}
Let $D$ be an arc diagram (of a word or a pair of words).
We denote the set of gates of $D$ by $G(D)$,
the oriented proper arcs on $D$ by $A(D)$,
and the arc in $A(D)$ from a gate $x$ to a gate $y$ by $(xy)$.
We consider a chain complex obtained from $D$ as follows:
The degree $0$ chain group $C_{0}(D)$ is a free $\mathbb{Z}$-module
generated from $G(D)$
and the degree $1$ chain group $C_{1}(D)$ is a free $\mathbb{Z}$-module
generated from $A(D)$.
The chain groups of degree more than $2$ are all $0$.
The boundary map $\partial=\partial_{1}: C_{1}(D) \rightarrow C_{0}(D)$ is defined by
$\partial(xy) = y - x$
(with linearly expanding on $C_{1}(D)$),
and the boundary maps of other degree are all zero map.

On $C_{0}(D)$
we give
a bi-linear form
directly associated with the order of the gates on $\partial_{0}D$.
\begin{dfn}\label{dfn:cdot}
 A bi-linear alternating form $\cdot : C_{0}(D)\times C_{0}(D) \rightarrow \mathbb{Z}$
 is defined by
 \begin{equation*}
  x\cdot y
   :=\begin{cases}
      +1 & (x < y) \\
      0 & (x = y) \\
      -1 & (x > y) 
     \end{cases},~
     \forall x,y \in G(D).
 \end{equation*}
\end{dfn}
\begin{rmk}
 \begin{enumerate}[i)]
  \item We may apply the operation $\cdot$ to the letters of a word $w$
        (and their inverses) directly
        without bringing up the arc diagram of the word.
        Namely we may define the operation $\cdot$ as follows:
        For the $i$-th letter $x_{i}$ and the $j$-th letter $x_{j}$ ($1 \leq i < j \leq p$)
        of a word $w=x_{1}x_{2}\dots x_{p}$,
        we define
         \begin{equation*}
          x_{i}^{\epsilon}\cdot x_{j}^{\delta}
           = - x_{j}^{\delta} \cdot x_{i}^{\epsilon}
           =\begin{cases}
              +1 & (x_{i}^{\epsilon} < x_{j}^{\delta}
             \text{~or~} x_{i}^{\epsilon} = x_{j}^{\delta}
              \in \{a_{k}, b_{k}\}_{1 \leq k \leq g}) \\
              -1 & (x_{i}^{\epsilon} > x_{j}^{\delta}
             \text{~or~} x_{i}^{\epsilon} = x_{j}^{\delta}
              \in \{a_{k}^{-1}, b_{k}^{-1}\}_{1 \leq k \leq g}) \\
             \end{cases},
         \end{equation*}
        where $\epsilon$ and $\delta$ are $\pm 1$,
        and the inequalities are of the order of $\mathcal{A}$.
  \item On the definition of the operation $\cdot$
        we may assign any value to $x\cdot x$,
        since the operation $\cdot$ is not performed on the same two gates
        throughout this paper.
 \end{enumerate}
\end{rmk}
 We now introduce the notion of the linking number of a pair of proper arcs on $D$,
 which define a bi-linear alternating form on $C_{1}(D)$.
\begin{dfn}
 The linking number $\operatorname{lk}(\alpha, \beta)$
 of $\alpha$ and $\beta$ in $A(D)$
 is defined as
 \begin{equation*}
  \operatorname{lk}(\alpha, \beta)
   = \frac{1}{2}\partial{\alpha}\cdot\partial{\beta}.
 \end{equation*}
\end{dfn}
This defines
the linking number of a pair of partitions of a word
through the identification between the proper arcs of the arc-diagram
and the partitions of the word.
Namely for the $i$-th partition $\phi_{i}w$ and the $j$-th partition $\phi_{j}w$
of a word $w=x_{1}\dots x_{p}$
we may
define the linking number $\operatorname{lk}(\phi_{i}w, \phi_{j}w)$
as
\begin{equation*}
 \operatorname{lk}(\phi_{i}w, \phi_{j}w)
  = \frac{1}{2}\partial\phi_{i}w{\cdot}\partial\phi_{j}w
  = \frac{1}{2}(x_{i+1}-x_{i}^{-1})\cdot (x_{j+1}-x_{j}^{-1}).
\end{equation*}

\begin{exmp}
 For a word $w=a_{1}b_{1}^{-1}a_{2}a_{1}$,
 \begin{eqnarray*}
  \operatorname{lk}(\phi_{1}w, \phi_{3}w)
  &=& \frac{1}{2}(b_{1}^{-1}-a_{1}^{-1})\cdot(a_{1}-a_{2}^{-1}) \\
   &=& \frac{1}{2}(b_{1}^{-1}\cdot a_{1}-b_{1}^{-1}\cdot a_{2}^{-1}
   -a_{1}^{-1}\cdot a_{1} + a_{1}^{-1}\cdot a_{2}^{-1}) \\
  &=& \frac{1}{2}((-1)-(+1)-(-1)+(+1)) 
  ~=~ 0,
 \end{eqnarray*}
 and
 \begin{eqnarray*}
  \operatorname{lk}(\phi_{3}w, \phi_{4}w)
  &=& \frac{1}{2}(a_{1}-a_{2}^{-1})\cdot (a_{1}-a_{1}^{-1})\\
  &=& \frac{1}{2}(a_{1,(4)}\cdot a_{1,(1)}-a_{1}\cdot a_{1}^{-1}
   -a_{2}^{-1}\cdot a_{1} + a_{2}^{-1}\cdot a_{1}^{-1}) \\
  &=& \frac{1}{2}((-1)-(+1)-(-1)+(-1)) 
  ~=~ -1,
 \end{eqnarray*}
 where the subscript $(4)$ in $a_{1,(4)}$ indicates that
 this $a_{1}$ is the $4$-th letter of $w$,
 and the same for $a_{1,(1)}$.
\end{exmp}

\begin{prop}\label{prp:linking}
 The linking number of a pair of proper arcs on an arc-diagram
 can take one of the values of $-1,0$ or $1$ and
 it is equal to the algebraic intersection number of the pair.
\end{prop}
\begin{proof}
 Let $D$ be an arc-diagram and $x,y,z,w$ distinct four gates in $G(D)$.
 We may assume that two proper arcs $(xy)$ and $(zw)$ have
 at most one intersection point.
 We will show that
 the value $\frac{1}{2}\partial(xy)\cdot \partial(zw)$ is equal to
 the sign of the intersection point of the arcs if they intersect
 and is equal to $0$ otherwise.

 Setting $X = x\cdot w-x\cdot z$ and $Y = y \cdot w - y \cdot z$,
 we have $\frac{1}{2}\partial(xy)\cdot\partial(zw) = \frac{1}{2}\left(X-Y\right)$.
 Note that
 \begin{equation*}
   X =
   \begin{cases}
    -2 & (w < x < z) \\
    +2 & (z < x < w) \\
    0 & (\text {otherwise}) 
   \end{cases},~~
  Y =
   \begin{cases}
    -2 & (w < y < z) \\
    +2 & (z < y < w) \\
    0 & (\text {otherwise}) 
   \end{cases}.
 \end{equation*}
 In the case where $w < x < z$ and $w < y < z$,
 the linking number
 $\frac{1}{2}\partial(xy)\cdot\partial(zw)
 =\frac{1}{2}\left(-2-(-2)\right)=0$
 and the arcs $(xy)$ and $(zw)$ have no intersection.
 It is the same in the case where $z < x < w$ and $z < y < w$.
 The case where $w < y < z$ and $z < x < w$
 and the case where $z < y < w$ and $w < x < z$ never occur.
 In the other cases, it is also straightforward to see that the linking number
 and the sign of the intersection of the arcs match.
\end{proof}

\section{Combinatorial Goldman-Turaev Lie bialgebra}\label{sec:CGT}
In this section we give new combinatorial definitions of
the Turaev cobracket and Goldman bracket.
For $w \in W(\mathcal{A})$, we denote by
$\conj{w}$ the conjugacy class of the element in $\pi=W(\mathcal{A})/\text{contraction}$
represented by $w$,
which is identified with the element of $\hat{\pi}$ represented by
the (free) loop $\ell(w)$ on $\Sigma$ mentioned in \S \ref{sec:linking}.

\subsection{Turaev cobracket}
For $w \in W(\mathcal{A})$ of length $p$,
let $w_{s,t}$ denote its sub-word
from the $s$-th letter to the $t$-th letter if $s \leq t$,
 and define $w_{s,t}$ as $w_{s,p}w_{1,t}$ if $s > t$.
 The index numbers are considered in modulo $p$.
\begin{dfn}
 We define a map
 $\delta: W(\mathcal{A}) \rightarrow \mathbb{Z}\hat\pi \otimes \mathbb{Z}\hat\pi$
 as follows:
 For $\forall w \in W(\mathcal{A})$ of length $p$,
 if $p=1$ we set $\delta(w)=0$
 and otherwise
 \begin{equation*}
  \delta(w)
 = \sum_{1 \leq i < j \leq p}
    \operatorname{lk}(\phi_{i}w, \phi_{j}w)
     \left(
     \conj{w_{i+1,j}}\otimes \conj{w_{j+1,i}} - \conj{w_{j+1,i}} \otimes \conj{w_{i+1,j}}
     \right).
 \end{equation*}
\end{dfn}
 \begin{thm}\label{thm:cobracket}
  For any word $w$ in $W(\mathcal{A})$,
  $\delta(w)=\Delta(|\ell(w)|)$.
 \end{thm}
 \begin{proof}
  Note that all self-intersection points of the loop $\ell(w)$
  of a word $w \in W(\mathcal{A})$
  are displayed in the arc diagram of $w$.
  Let $\ell(w)^{1}_{x}$ and $\ell(w)^{2}_{x}$ be the oriented loops based
  at $x \in \Gamma(\ell(w))$
  obtained by dividing $\ell(w)$ at $x$ so that
  $\ell(w)^{1}_{x}$ is freely homotopic to $\ell(w_{i+1,j})$ and
  $\ell(w)^{2}_{x}$ is to $\ell(w_{j+1,i})$.
  We can calculate $\Delta(|\ell(w)|)$ as
  \begin{equation*}
   \Delta(|\ell(w)|) = \sum_{x \in \Gamma(\ell(w))}
    \operatorname{sign}(\ell(w)^{1}_{x}, \ell(w)^{2}_{x})
    (|\ell(w)^{1}_{x}|\otimes|\ell(w)^{2}_{x}|
    -|\ell(w)^{2}_{x}|\otimes|\ell(w)^{1}_{x}|).
  \end{equation*}
  It is immediate from Proposition \ref{prp:linking} that
  \begin{equation*}
   \operatorname{sign}(\ell(w)^{1}_{x}, \ell(w)^{2}_{x})
    = \operatorname{lk}(\phi_{i}w, \phi_{j}w)
  \end{equation*}
  for the pair of partitions $(\phi_{i}w, \phi_{j}w)$ such that they intersect at $x$,
  and one can see that the (free) homotopy class of $\ell(w)^{1}_{x}$ and $\ell(w)^{2}_{x}$
  are respectively identified with $|w_{i+1,j}|$ and $|w_{j+1,i}|$.
  We also know from Proposition \ref{prp:linking} that
  the value $\operatorname{lk}(\phi_{i}w, \phi_{j}w)=0$
  for the pairs of partitions of $w$ which has no intersections.
  This completes the proof.
 \end{proof}
 It follows from Theorem \ref{thm:cobracket} that
 the map $\delta$ induces a map from the set of cyclic words on $\mathcal{A}$
 to $\mathbb{Z}\hat\pi/\mathbb{Z}1 \otimes \mathbb{Z}\hat\pi/\mathbb{Z}1$
 and gives a new combinatorial definition of Turaev cobracket
 through identifying the set of cyclic words with $\hat{\pi}$
 and linearly expanding $\delta$ onto $\mathbb{Z}\hat{\pi}$.
\begin{rmk}\label{rmk:well-definedness}
 From our definition of $\delta$ we may directly show the well-definedness
 of $\delta$ as a map from the set of cyclic words on $\mathcal{A}$
 to $\mathbb{Z}\hat\pi/\mathbb{Z}1 \otimes \mathbb{Z}\hat\pi/\mathbb{Z}1$
 as follows:

 We first note that one can see the invariance of $\delta(w)$
 under the cyclic permutation by the following simple observation.
 Let $w'$ be a word given from a word $w=x_{1}x_{2}\dots x_{p}$ by one cyclic permutation,
 i.e., $w'=x_{2}\dots x_{p}x_{1}$.
 Then the partition $\phi_{1}w$ of $w$ turns into the partition $\phi_{p}w'$ of $w'$.
 This causes the inversion of coefficients of terms in $\delta(w)$
 which are given by $\phi_{1}w$ with other partitions
 i.e.,
 $\partial\phi_{1}w{\cdot}\partial\phi_{i}w
 = (-1) \partial\phi_{i-1}w'{\cdot}\partial\phi_{p}w'$,
 while the wedge part $w_{1i} \wedge w_{i1}$ of the terms
 turn into $w'_{i-1,p} \wedge w'_{1i-1}= - w_{1i} \wedge w_{i1}$.
 Therefore $\delta(w')=\delta(w)$.

 Now all that remains is to show $\delta(x^{-1}wx)=\delta(w)$
 for $\forall x \in \mathcal{A}$ and
 $\forall w=x_{1}x_{2}\dots x_{p} \in W(\mathcal{A})$.
 We denote $x^{-1}wx$ by ${}^{x}w$
 and identify $|u| \otimes |v| - |v| \otimes |u|$
 with $|u| \wedge |v|$
 for short.
 Note the following three simple facts in advance;
 \begin{enumerate}[i)]
  \item $\partial\phi_{p+2}({}^{x}w) = x^{-1}-x^{-1}=0$,
  \item $|w| \wedge |xx^{-1}| = 0$ in $\mathbb{Z}\hat{\pi}/\mathbb{Z}1$,
        because of $|xx^{-1}|=1$,
  \item $\partial\phi_{1}({}^{x}w)+\partial\phi_{p+1}({}^{x}w)
        = (x_{1}^{-1}-x^{-1})+(x^{-1}-x_{p})
        = \partial\phi_{p}w$.
 \end{enumerate}
 Using these facts we have
 \begin{eqnarray*}
  \delta({}^{x}w)
  &=&
  \sum_{2 \leq j \leq p}
  \operatorname{lk}(\phi_{1}({}^{x}w),\phi_{j}({}^{x}w)) 
  |w_{1,j-1}| \wedge |w_{j,p}xx^{-1}| \\
  && + \operatorname{lk}(\phi_{1}({}^{x}w), \phi_{p+1}({}^{x}w)) 
   |w|\wedge |xx^{-1}|
   + \sum_{2 \leq i \leq p}
   \operatorname{lk}(\phi_{i}({}^{x}w), \phi_{p+1}({}^{x}w)) 
       |w_{i,p}| \wedge |xw_{1,i-1}x^{-1}|
        \\
  && + \sum_{2 \leq i < j \leq p}
  \operatorname{lk}(\phi_{i}({}^{x}w), \phi_{j}({}^{x}w))
  |w_{i,j-1}|\wedge |w_{j,p}xx^{-1}w_{1,i-1}| \\
  && +\sum_{1 \leq i \leq p+1}
  \operatorname{lk}(\phi_{i}({}^{x}w), \phi_{p+2}({}^{x}w)) 
   |w_{i,p}x| \wedge |x^{-1}w_{1,i-1}| \\
  &=& \sum_{2 \leq i \leq p}
  \frac{\partial\phi_{i-1}w{\cdot}
  \left\{
   \partial\phi_{1}({}^{x}w)+\partial\phi_{p}({}^{x}w)
  \right\}}{2}
  |w_{i,p}| \wedge |w_{1,i-1}| \\
  && + \sum_{2 \leq i < j \leq p}
   \operatorname{lk}(\phi_{i-1}w, \phi_{j-1}w)
  |w_{i,j-1}|\wedge |w_{j,i-1}| \\ 
  &=& \delta(w).
 \end{eqnarray*}
\end{rmk}
\subsection{Goldman bracket}

Let $\nu$ be the cyclic permutation on the words in $W(\mathcal{A})$
given by $\nu(x_{1}x_{2}\dots x_{p})=x_{2}\dots x_{p}x_{1}$ ($x_{i} \in \mathcal{A}$).
\begin{dfn}
 We define a map
 $\braket{}: W(\mathcal{A}) \times W(\mathcal{A}) \rightarrow \mathbb{Z}\hat\pi$
 as follows:
 For any words $v$ and $w$ in $W(\mathcal{A})$,
 we set
 \begin{equation*}
  \braket{v,w}
   = \sum_{1 \leq i \leq p}\sum_{1 \leq j \leq q}
   \operatorname{lk}(\phi_{i}v,\phi_{j}w)
   \conj{\nu^{i}(v)\nu^{j}(w)}.
 \end{equation*}

\begin{thm}\label{thm:bracket}
  For any two words $v$ and $w$ in $W(\mathcal{A})$,
  $\braket{v,w}=[|\ell(v)|,|\ell(w)|]$.
 \end{thm}
 \begin{proof}
  The all intersection points of $\ell(v)$ with $\ell(w)$ are displayed
  in the arc diagram of the pair $v$ and $w$
  as the intersection of proper arcs in $\{\phi_{i}(v)\}_{1 \leq i \leq p}$ with
  proper arcs in $\{\phi_{j}(w)\}_{1 \leq j \leq q}$.
  The rest of this proof is the same as the proof of Theorem \ref{thm:cobracket}.
 \end{proof}
\end{dfn}

In the same way as 
in Remark \ref{rmk:well-definedness},
we may see the well-definedness of $\braket{}$
as the map
from the direct product of the set of cyclic words with itself
to $\mathbb{Z}\hat{\pi}$
directly from our definition.



\end{document}